\documentclass[12pt, titlepage]{article}
\usepackage{eurosym}
\usepackage{amsmath,amssymb,amsthm,amsfonts}
\usepackage{epstopdf}
\usepackage{subcaption}
\usepackage{appendix}
\usepackage{amsmath}
\usepackage{amsfonts}
\usepackage{amssymb}
\usepackage{multirow}
\usepackage{rotating}
\usepackage{xcolor}
\usepackage{lscape}
\usepackage{verbatim}
\usepackage{graphicx}
\usepackage{cancel}
\usepackage[normalem]{ulem}
\usepackage{url}
\numberwithin{equation}{section} 

\newtheorem{theorem}{Theorem}[section]
\newtheorem{corollary}{Corollary}[section]

\newtheorem{proposition}{Proposition}[section]

\newtheorem{remark}{Remark}
\newtheorem{example}{Example}

\def\XX{\boldsymbol{X}}
\def\xx{\boldsymbol{x}}

\def\rr{\boldsymbol{r}}

\def\RRR{\boldsymbol{R}}

\def\FFF{\mathcal{F}}

\def\EEE{\mathcal{E}_d}

\def\PPP{\mathcal{P}}

\def\t{\theta}

\def\NN{\mathbb{N}}

\def\RR{\mathbb{R}}
\def\aa{\boldsymbol{a}}

\def\m{\mathcal}

\def\ff{\boldsymbol{f}}
\def\sss{\boldsymbol{\sigma}}

\def\design{\mathcal{X}}
\def\DDD{\mathcal{D}_d}

\def\argmax{\text{argmax}}
\def\argmin{\text{argmin}}

\begin{document}

\author{ R.FONTANA  \\ \textit{\ %EndAName
Department of Mathematical Sciences G. Lagrange,} \\ {Politecnico di
Torino.}
\\ P. SEMERARO\footnote{%
Corresponding author: Patrizia Semeraro
Department of Mathematical Sciences G. Lagrange, Politecnico di
Torino. Email:patrizia.semeraro@polito.it} \\ \textit{\ %EndAName
Department of Mathematical Sciences G. Lagrange,} \\ {Politecnico di
Torino.}}
\title{Multivariate Bernoulli distributions behind a discrete distribution: the binomial case}
\title{ Bernoulli sums and  the (Poisson-) binomial ditribution}
\title{The Bernoulli structure of discrete distributions}
\maketitle

\begin{abstract}

Any discrete distribution with support on $\{0,\ldots, d\}$ can be constructed as the distribution of sums of Bernoulli variables. We prove that
the class of $d$-dimensional Bernoulli variables $\XX=(X_1,\ldots, X_d)$  whose sums $\sum_{i=1}^dX_i$ have the same distribution $p$  is  a convex polytope $\PPP(p)$ and we analytically find its extremal points.
Our main result is  to prove that the
Hausdorff measure of the polytopes $\PPP(p), p\in \DDD,$ is a continuous  function $l(p)$ over $\DDD$ and it is the density of a finite measure $\mu_s$  on $\DDD$ that is  Hausdorff absolutely continuous. We  also prove that the  measure $\mu_s$ normalized over the simplex $\DDD$ belongs to the class of Dirichlet distributions. We observe that the symmetric binomial distribution is the mean of the Dirichlet distribution on $\DDD$ and that when $d$ increases it converges to the mode.

\noindent \textbf{Keywords}: multidimensional Bernoulli distribution, Dirichlet distribution, binomial distribution, extremal points, polytope.
\end{abstract}

\section{Introduction}

Sums of Bernoulli random variables model the number of occurrences of some events within $d$ repeated trials.  The  case of  $d$ independent and identically distributed Bernoulli variables, the sum of which follows the binomial distribution,  is often used in modelling across different areas, such as reliability  (e.g. \cite{martz1988bayesian}) and finance (e.g. \cite{fontana2020model}). However, the binomial distribution also  arises from sums of dependent Bernoulli variables in many ways (\cite{vellaisamy2001nature}), making it a possible model even when independence cannot be assumed,  \cite{van2005binomial}. Actually, any discrete distribution with support on $\{0,\ldots, d\}$ can be constructed as the distribution of sums of Bernoulli variables in many ways (\cite{fontana2020model}).
Formally, let $\DDD\subset \RR^{d+1}$ be the $d$-simplex of discrete probability mass functions  on $\{0,\ldots, d\}$ and $\FFF_d\subset\RR^{2^d}$ be the $2^d-1$-simplex  of $d$-dimensional Bernoulli probability mass functions.
For any $p\in \DDD$, we define the class $\PPP(p)$ of probability mass functions  $\ff\in\FFF_d$ such that if $\XX=(X_1,\ldots, X_d)$ has probability mass function  $\ff$ then $\sum_{i=1}^dX_i$ has probability mass function $p$.

In \cite{chevallier2011law}, the author proves that, as the dimension $d$ increases, the normalized Hausdorff measure of  Bernoulli sums with distribution  close to the symmetric  binomial distribution $Bin(1/2, d)$ converges to one. This means that  the Bernoulli sums far from the binomial distributions are rare. It can be asked whether
if this is related to the  Hausdorff measure  of  $\PPP(b(1/2))$, where $b(1/2)$ is the probability mass function of $Bin(1/2, d)$   compared to the Hausdorff measure of $\PPP(p)$ for any other $p\in\DDD$ .
Inspired by this question we characterize the class $\PPP(p)$   for any $p\in\DDD$. 
We prove that for any $p\in\DDD$, $\PPP(p)\subset \RR^{2^d-d-1}$ is a convex polytope and we analytically find its extremal points. The geometrical structure of $\PPP(p)$ allows us to analytically  study many  statistical properties of dependent  Bernoulli trials with a given distribution of their sum.
We find analytical bounds in $\PPP(p)$ for the (cross)  moments of  the Bernoulli variables, that are important in modelling dependence \cite{chaganty2006range}  and  for the Shannon entropy, that is widely studied for sums of Bernoulli variables (see \cite{shepp1981entropy}, \cite{hillion2017proof}, and \cite{hillion2019proof}). 

Our main result goes a  step further.  We prove that the
Hausdorff measures of the polytopes $\PPP(p), p\in \DDD$ define a continuous function $l(p)$ over $\DDD$ which is the density of an Hausdorff-absolutely continuous, positive, and  finite measure $\mu_s$ on $\DDD$. We  also prove that the normalized measure $\mu_s$ belongs to the class of Dirichlet distributions. The Dirichlet parameters are linked to the dimension of the polytopes $\PPP(p)$. This means that we have a geometrical interpretation of the Dirichlet distribution for  a specific choice of  its parameters. 

Finally, we focus on the polytope of the binomial distribution $Bin(\t, d)$. 
We observe that the symmetric binomial distribution is the mean of the Dirichlet distribution on $\DDD$ and that when $d$ increases it converges to the mode. This answers our question: we have the density $l(p)$ for any $p$ and find that the Binomial probability mass function $b(1/2)$ is close to its  maximum value.   For any dimension $d$, given $l(p)$ we can also find the Hausdorff measure $\mathcal{H}^d$  of a neighborhood of $\PPP(b)$ in $\DDD$, and therefore we can measure the size of  probability mass functions  of Bernoulli variables  with symmetric binomial  sums even for low dimensions when the asymptotic result in \cite{chevallier2011law} does not applies.

%We also have other interesting results, we find that the the polytope containing the Bernoulli distribution with maximal entropy in $\FFF_d$ correspond to the binomial distribution.

This paper is organized as follows. Section \ref{Sec:cxPol} introduces the class $\PPP(p)$, proves that it is a polytope, find its  extremal points and studies its properties. Our main result is in Section \ref{Sec: mesure}, where we find the Hausdorff measure of $\PPP(p), \,\,\, p\in \DDD$ and we prove that it defines a density on $\DDD$. Finally,  Section \ref{Sec:Binom} focuses on the Binomial distribution and answers our original question.

\section{The convex polytope $\PPP(p)$}\label{Sec:cxPol}
Let $\design=\{0,1\}^d$, we make the non-restrictive hypothesis that the set $\design$ of $2^d$ binary vectors is ordered according to the reverse-lexicographical criterion.
For example for $d=3$, $\design=\{000, 100, 010, 110, 001, 101, 011, 111\}$.
 Let $\XX=(X_1,\ldots, X_d)$ be a $d$-dimensional Bernoulli  random variable with probability mass function (pmf) $f$. We identify $f$ with the column vector which contains the values of $f$ over $\design=\{0, 1\}^d$,  by $\ff = (f_1,\ldots, f_{2^d})=(f_{\xx}:\xx\in\design):=(f(\xx):\xx\in\design)$.
Let $\FFF_d$ the ${2^d-1}$-simplex of $d$-dimensional pmfs $\ff$ of Bernoulli vectors. In this paper we identify random variables with their distributions, therefore the notation $\XX\in \FFF_d$ means that $\XX$ has pmf $\ff\in \FFF_d$.

Let $\DDD$ be the $d$-simplex of discrete pmfs on $\{0,\ldots, d\}$. The notation $D\in \DDD$ means that $D$ has pmf $p=(p_0,\ldots, p_d)\in \DDD$.

Any pmf $p\in \DDD$ is the distribution of the sum of  the components of  at least one $d$-dimensional Bernoulli random vector $\XX\in \FFF_d$ (see e.g. \cite{fontana2020model}). Actually, in general, behind any discrete pmf there are infinite Bernoulli vectors $\XX\in \FFF_d$. Formally, we define the following map between $\FFF_d$ and $\DDD$.
\begin{equation}\label{funsum}
\begin{split}
s:&\FFF_d\rightarrow \DDD\\
   &  \ff\rightarrow p_{\ff},
\end{split}
\end{equation}
where $p_{\ff}:=s(\ff)$ is the distribution of $S:=\sum_{i=1}^dX_i$ and $\XX\sim \ff$, i.e. $\XX$  has pmf $\ff\in\FFF_d$.
For any $p\in \DDD$, we define 
\begin{equation}\label{plytope}
\PPP(p)=\{\ff\in \FFF_d: \, p_{\ff}=p\}.
\end{equation} 

 The next Theorem \ref{PolGen} proves that for any choice of $p$, the class $\PPP(p)$ is a convex polytope and  provides an analytical expression for its extremal pmfs.  

We need some notation first.
Let $\design_k=\{\xx\in \design: \sum_{i=1}^dx_i=k\}$ be the subset of $\design$ that contains all the $\binom{d}{k}$ binary vectors with $k$ ones and $d-k$ zeros, $k=0, 1, \ldots, d$.
We observe that $\design_k$ inherits the order of $\design$. Let $\xx_k^j$ be the j-th elemnt of $\design_k$, $j=1,\ldots \binom{d}{k}$.  The first element is $\xx_k:=\xx_k^1=(\underbrace{1,\ldots,1}_{k},0,\ldots, 0)$.  We need to  introduce the simplexes $\Delta_{n,\sqrt{2}p}$ with dimension $n$ and  side $\sqrt{2}p$
$$\Delta_{n,\sqrt{2}p}=\{\xx\in \RR^{n+1}: x_j\geq 0, j=0,\ldots, n, \, \sum_{h=0}^{n}x_h=p\}.$$ 
	
\begin{theorem}\label{PolGen}
For any $p\in \DDD$ the class $\PPP(p)=\{\ff\in \FFF_d: \, p_{\ff}=p\}$ is the  convex polytope 
$\PPP(p)=\prod_{k=0}^d\Delta_{n_k, \sqrt{2}p_k}$, where $n_k=\binom{d}{k}-1$
 and its extremal points are 
\begin{eqnarray}\label{sol}
f^{\sss}(\xx)=\begin{cases}
%p_0 \,\,\, \text{if}\,\,\, \xx=(0,\ldots, 0) \\
p_k \,\,\, \text{if}\,\,\, \xx=\xx_k^{\sigma_k}\\
%p_d  \,\,\, \text{if} \,\,\,\xx=(1,\ldots, 1)\\
0 \,\,\, \text{otherwise},
\end{cases}
\end{eqnarray}
where  $\sigma=(\sigma_0, \ldots, \sigma_k, \ldots, \sigma_d)$, $\sigma_k=1,\ldots, \binom{d}{k}$, $k=0,\ldots, d$.
\end{theorem}
\begin{proof}
Let $p\in \DDD$ and let $\XX\sim \ff\in \FFF_d$. We have  $p_{\ff}=p$ if and only if ${\ff}$ is a  positive solution of the linear system:
\begin{eqnarray}\label{sistem*}
%\begin{cases}
\sum_{\xx\in \design_k}f(\xx)=p_k, \,\,\, k=0,\ldots d.
%\end{cases}
\end{eqnarray}
%\begin{eqnarray}\label{sistem*}
%%\begin{cases}
%\sum_{\xx\in \design_k}(1-p_k)f(\xx)-\sum_{\xx\in \bar{\design}_k}p_kf(\xx)=0, \,\,\, k=0,\ldots d.
%%\end{cases}
%\end{eqnarray}
Each equation of the system  \eqref{sistem*}  defines a $\binom{d}{k}-1$-simplex with side $\sqrt{2}p_k$.
It is well known that the $\binom{d}{k}$ extremal points of the simplex are $(p_k,0,\ldots,0)$, $(0, p_k,\ldots,0)$, $\ldots$, $(0,\ldots,0,p_k)$.
Since $\design_k\cap \design_j=\emptyset$, for any $k\neq j$, the extremal solutions of the system are the ones in \eqref{sol}. 
%The equations \eqref{sistem*} are a linear sistem  of $d$ equations and $\PPP(p)$ is the convex polytope
%\begin{equation}  \label{cone}
%\PPP(p)=\{\ff\in \RR^{2^d}: H\ff=0, f_i\geq 0, \sum_{i=0}^{2^d}f_i=1\},
%\end{equation} where $H$ is the coefficient matrix of \eqref{sistem*}. 
%By Proposition 4.3  in \cite{fontana2020model}  the extremal points of this polytope  have support on at most $d+1$ points. 
%It is straightforward to verify that $f^{\sigma}\in \PPP(p)$ and that it has support on at most $d+1$ points.
%
%Let $\tilde{p}\in \PPP(p)$ and consider two cases.\\
%Case 1: $p$ has full support. If $\tilde{\ff}$ is a solution of \eqref{sistem*} that $\tilde{\ff}$ must have support on at least one point $\xx\in \design_k$, $k=0,\ldots, d$, and since it has support on at most $d+1$, this implies that it has support on exactly one $\tilde{\xx}\in\design_k$, $k=0,\ldots, d$. We have $\tilde{\xx}=\sigma(\xx_k)$ for a $\sigma\in C^d_k$. This implies $f(\tilde{\xx})=p_k$ and $\tilde{\ff}=\ff^{\sigma}$.\\
%Case 2:  $p$ has support on $k=k_1,\ldots, k_n$. The system \eqref{sistem*} reduces to the $n$ equations with $k=k_1,\ldots, k_n$, and the solution of $\PPP(p)$ in this case have support on $n+1$ points. The argument of case 1 applies accordingly.

\end{proof}
\begin{corollary}\label{npoints}
The number $n_p$ of extremal points of $\PPP(p)$ is
\begin{equation*}
n_p=\prod_{k\in \text{Supp}(p)} \binom{d}{ k},
\end{equation*}
where $ \text{Supp}(p)\subseteq\{0,\ldots, d\}$ is the support of $p$.
\end{corollary}
\begin{proof}
The proof follows from Theorem \ref{PolGen}  since  $\#\design_k=\binom{d}{k}$.
\end{proof}
From Theorem \ref{PolGen} and Corollary \ref{npoints}  it follows that for any $\ff\in \PPP(p)$
there exist $\lambda_i\geq0$ summing up to one and   $\rr_{i}\in \PPP(p)$, $i=1,\ldots, n_p$ such that
\begin{equation*}
\ff=\sum_{i=1}^{n_p}\lambda_i\rr_{i}.
\end{equation*}

We call  $\rr_{i}$ extremal points or  extremal pmfs of $\PPP(p)$. We denote with $\RRR_i$ a $d$-dimensional  random variable  with distribution $\rr_{i}$.

Notice that $n_p$ depends only on the support and not on the values of $p$. If $\{1,\ldots, d-1\}\subseteq \text{Supp}(p)$, since $\binom{d}{0}=\binom{d}{d}=1$ we have
\begin{equation*}
n_p=\prod_{k=0}^d \binom{d}{ k}.
\end{equation*}

\begin{example}\label{example}
As an illustrative example we consider the polytope $\PPP(p)$ in dimension $d=3$ for a given $p=(p_0, p_1,p_2, p_3)\in \mathcal{D}_3$ with full support.

The  extremal points of $\PPP(p)$ are $n_{p}=\binom{3}{1}\binom{3}{2}=9$ and   they are reported in Table \ref{tab:d}.

\begin{table}[h]
	\centering
		\begin{tabular}{ccc|rrrrrrrrr}
$\xx_1$ &	$\xx_2$ &	$\xx_3$ &	$\rr_1$ & 		$\rr_2$ & 		$\rr_3$ & 		$\rr_4$ & 		$\rr_5$ & 		$\rr_6$ &$\rr_7$&$\rr_8$& $\rr_9$\\
		\hline
0 &	0 &	0 &	$p_0$ &	$p_0$ &	$p_0$&	$p_0$ &	$p_0$&	$p_0$ &$p_0$&$p_0$&$p_0$\\
1 &	0 &	0 &	$p_1$ &	$p_1$ &	$p_1$ &	0 &	0 &	0&0&0&0\\
0 &	1 &	0 &	0 &	&	0 &	$p_1$ &	$p_1$  &	$p_1$ &0&0&0\\
1 &	1 &	0 &$p_2$ &	0 &	0 &	$p_2$&	0 &	0&0&0&$p_2$\\
0 &	0 &	1 &	0&	0&	0 &	0&	0 &	0&$p_1$ &$p_1$ &$p_1$ \\
1 &	0 &	1 &	0 &	$p_2$ &0 &	0&	$p_2$&	0 &$p_2$&0&0\\
0 &	1 &	1 &	0 &	0 &$p_2$&	0 &	0 &	$p_2$&0&$p_2$&0\\
1 &	1 &	1 &	$p_3$ &	$p_3$  &	$p_3$  &	$p_3$  &	$p_3$ &	$p_3$ &$p_3$  &$p_3$ &$p_3$ \\
		\end{tabular}
	\caption{Extremal pmfs  of $\PPP(p)$, case $d=3$.}
	\label{tab:d}
\end{table}

\end{example}
\begin{remark}\label{RemGen}
Theorem \ref{PolGen} can be straightforward generalized to any surjective map $h:\design\rightarrow \{0,\ldots, d\}$. Let $p^h_{\ff}\in \DDD$ the pmf associated to ${h}(\XX)$ with $\XX\sim \ff\in \FFF_d$.  For any $p\in \DDD$ the class $\mathcal{P^H}(p)=\{\ff\in \FFF_d: \, p^h_{\ff}=p\}$ is a convex polytope and its extremal points are 
\begin{eqnarray*}
f^{\sss}(\xx)=\begin{cases}
%p_0 \,\,\, \text{if}\,\,\, \xx=(0,\ldots, 0) \\
p_y \,\,\, \text{if}\,\,\, \xx=\xx_y^{\sigma_y},\\
%p_d  \,\,\, \text{if} \,\,\,\xx=(1,\ldots, 1)\\
0 \,\,\, \text{otherwise}, 
\end{cases}
\end{eqnarray*}
where for any $y\in\{0,\ldots, d\}$,  $\xx_y^{\sigma_y}$  is the $\sigma_y$-th element  of $h^{-1}({y})$.
If $h(\xx)=\sum_{i=1}^dx_i$ we have Theorem \ref{PolGen}.

\end{remark}
\subsection{Moments}
In \cite{fontana2018representation}, the authors prove that the bounds of the moments of a pmf in a convex polytope are sharp and reached on the extremal points, that in this case are explicitely known.

\begin{proposition}\label{prop:mom}
Let $\XX\in \PPP(p)$,  then for any $ \{j_1,\ldots, j_k\}\subseteq \{1,\ldots, d\}$,
\begin{equation*}
p_d\leq E[X_{j_1}\cdots X_{j_k}]\leq \sum_{h=k}^{d}p_h, 
\end{equation*}
and the bounds are sharp.
\end{proposition}
\begin{proof}

We prove $ E[X_{j_1}\cdots X_{j_k}]\leq \sum_{h=k}^{d}p_h$, for any $ \{j_1,\ldots, j_k\}\subseteq \{1,\ldots, d\}$. Let $\rr$ be an extremal point in $\PPP(p)$.
%For any $\XX\sim f^{\sss}$, we have  $E[X_{j_1}\ldots X_{j_k}]=f(\sigma_k(\underbrace{1,\ldots,1}_{k},0,\ldots, 0)+f(\sigma_{k+1}(\underbrace{1,\ldots,1}_{k+1},0,\ldots, 0))+\ldots +f(\sigma_{d-1}(\underbrace{1,\ldots,1}_{d-1},0))+f(\sigma_d(\underbrace{1,\ldots,1}_{d}))\leq\sum_{h=0}^kp(k)$. If we consider $\XX\sim \tilde{r}$,
 For any $\RRR\sim \rr$, we have  $$E[R_{j_1}\ldots R_{j_k}]=\sum_{\{\xx\in \design: x_{j_1}=1,\ldots, x_{j_k}=1\}}r(\xx)\leq\sum_{h=k}^dp_h.$$

Consider now the extremal pmf
\begin{eqnarray*}
\tilde{r}(\xx)=\begin{cases}
%p_0 \,\,\, \text{if}\,\,\, \xx=(0,\ldots, 0) \\
p_k \,\,\, \text{if}\,\,\, \xx_k, \\
%p_d  \,\,\, \text{if} \,\,\,\xx=(1,\ldots, 1)\\
0 \,\,\, \text{otherwise},
\end{cases}
\end{eqnarray*}
and $\tilde{\RRR}\sim \tilde{r}$.
We have
$E[\tilde{R}_{1}\ldots \tilde{R}_k]=\tilde{r}(\xx_k)+r(\xx_{k+1})+\ldots +\tilde{r}(\xx_{d-1})+\tilde{r}(\xx_{d})=\sum_{h=k}^dp_h$.

We now prove $p_d\leq E[X_{j_1}\cdots X_{j_k}]$, for any $ \{j_1,\ldots, j_k\}\subseteq \{1,\ldots, d\}$.
For any $j_1,\ldots, j_k$ it holds
\begin{equation*}
 E[{R}_{j_1}\cdots {R}_{j_k}]= P({R}_{j_1}=1,\cdots,{R}_{j_k}=1)\geq P({R}_{1}=1,\cdots, {R}_{d}=1)= p_d.
\end{equation*}
If we consider $\tilde{\RRR}\sim \tilde{r}$,
we have
$E[\tilde{R}_{d-k+1}\ldots \tilde{R}_d]=p_d$.

\end{proof}

Proposition \ref{prop:mom} also gives a necessary condition on the means, that is
\begin{equation*}\label{means}
p_d\leq E[X]\leq 1-p_d.
\end{equation*}

The compatibility condition on the Bernoulli means is a necessary condition for a Frech\'et class of multivariate Bernoulli distribution to be compatible with a discrete distribution. Formally, let 
$\FFF(\theta)$ be the Fr\'echet class of Bernoulli distributions with given mean vector $\theta=(\theta_1,\ldots, \theta_d)$ and 
let $\PPP(p, \theta):=\PPP(p) \cap \FFF(\theta)$ be the intersection of the polytope $\PPP(p)$ with  $\FFF(\theta)$. A set of necessary conditions for   $\PPP(p, \theta)\neq \emptyset$  is
\begin{equation}\label{diseq}
\begin{split}
&\sum_{i=1}^d\t_i=\mu,\\
&p_d\leq \theta_i\leq 1-p_d,\,\,\, i=1,\ldots, d,
\end{split}
\end{equation}
where $\mu$ is the mean of the pmf $p$.
If the class $\PPP(p, \theta)$ is not empty
it is a polytope itself, since it is the set of positive normalized solution  of the following homogeneous linear system:
\begin{eqnarray*}\label{sistem*2}
\begin{cases}
\sum_{\xx\in \design_k}(1-p_k)f(\xx)-\sum_{\xx\in \bar{\design}_k}p_kf(\xx)=0,\,\,\, k=0,\ldots d\\
\sum_{\xx\in \design: x_k=1}(1-\t_k)f(\xx)-\sum_{\xx\in \design: x_k=0}\t_kf(\xx)=0\,\,\, k=1,\ldots d,
%(\gamma_k(\boldsymbol{1}-\yy_k)^{\top} - \yy_k^{\top}) f(\yy) = 0,\,\,\, k=1,\ldots d,\\
\end{cases}
\end{eqnarray*}
%where $\gamma_k=\theta_k/(1-\theta_k)$,  $\boldsymbol{1}$ is the vector with all the elements equal to $1$ and  $\yy_i$ is the vector which contains only the $i$-th element of $\yy \in \design$, $i\in\{1,\ldots,d\}$, e.g for the bivariate case $\yy_1=(0, 1,0,1)$ and $\yy_2=(0, 0,1,1)$.  
where $\bar{\design}_k=\design\setminus \design_k$.
\begin{remark}
For some choices of $p\in \DDD$ the condition \eqref{diseq} is also sufficient to have  $\PPP(p, \theta)\neq\emptyset$. For example, if $p$ has mean $\mu$  and we choose  $\t_i=\mu/d,\,\,\, i=1,\ldots,d$ we have
 $\PPP(p, \t)\neq\emptyset$. In fact $p\in \DDD(\mu)$, where $ \DDD(\mu)$ is the class of discrete distributions on $\{0,\ldots, d\}$ with mean $\mu$   and $\DDD(\mu)\leftrightarrow \EEE(\mu/d)$, the class of $d$ dimensional exchangeable Bernoulli distributions with means $\mu/d$. Therefore the unique exchangeable Bernoulli distribution associated to $p$ is an element of $\PPP(p, (\mu/d, \ldots, \mu/d))$, see \cite{fontana2020model}.
\end{remark}

The following example considers two non-degenerate discrete distributions with minimal support. The first provides a counterxample to show that the condition  on the mean is not in general sufficient for minimal support pmfs $p$ to have $\PPP(p,\theta)\neq \emptyset$, and the second one  considers a specific pmf $p$ that arises from a financial problem for which the necessary condition in Equation \eqref{diseq} is also sufficient to have $\PPP(p, \theta)\neq \emptyset$.

\begin{example}

The minimal condition to have $\PPP(p, \theta)\neq\emptyset $ is that $\sum_{i=1}^d\t_i=\mu$, where  $\mu$ is the expected value of  $p$.
This condition is not sufficient. In fact, let $d=3$ and 

\begin{equation*}  \label{binuleq}
p(y)=\left\{
\begin{array}{cc}
0.8 & y=0 \\
0.2 & y=3\\
0 & \text{otherwise}%
\end{array}
\right..
\end{equation*}
We have $\mu=0.6$. Let us consider $\t=(0, \,\,\, \t_2, \,\,\, \t_3)$, such that $\t_2+\t_3=0.6$. We have $\PPP(p, \theta)=\emptyset$ although the condition on the mean $\t_1+\t_2+\t_3=\mu$ is satisfyed. In fact, if $\XX\sim\ff\in \FFF(\theta)$ then  $\t_1=0$ implies $E[X_1]=0$ and therefore $P(X_1=1, X_2=1, X_3=1)=0$ that is not consistent with $p_3=0.2$.

Let $\DDD(\mu)$ be the class of discrete distributions with mean $\mu$.
In \cite{fontana2024high} and in  \cite{bernard2017robust} the authors prove that the minimal  element in convex order $p_{cx}^{\mu}$ (convex order is a partial order on a Fr\'echet class,  see \cite{shaked2007stochastic} for its definition and properties) of $\DDD(\mu)$ is:

\begin{equation*} 
p^{\mu}_{cx}(y)=\left\{
\begin{array}{cc}
\frac{{j_m}-\mu}{{j_m}-{j_M}} & y={j_M} \\
\frac{\mu-j_M}{{j_m}-{j_M}} & y={j_m} \\
0 & \text{otherwise}%
\end{array}
\right.,
\end{equation*}
with $j_M$ being the largest $j\in \{0,\ldots, d\}$ smaller than $\mu$ and $j_m$ the smallest $y\in  \{0,\ldots, d\}$ bigger than $\mu$. 
Let $\PPP^{\mu}_{cx}=\PPP(p^{\mu}_{cx})$ be the polytope of pmfs in $\FFF$ whose sums have  $p_{cx}^{\mu}$ distribution. From Theorem \ref{PolGen}
its  extemal points are 
\begin{eqnarray*}
f_{cx}^{\sigma}(\xx)=\begin{cases}
\frac{{j_m}-\mu}{{j_m}-{j_M}} \,\,\, \text{if}\,\,\, \xx=\xx^{\sigma_{j_M}}_{j_M}\\
\frac{\mu-j_M}{{j_m}-{j_M}} \,\,\, \text{if}\,\,\, \xx=\xx^{\sigma_{j_m}}_{j_m}\\
0 \,\,\, \text{otherwise}.
\end{cases}
\end{eqnarray*}

We have
  $\PPP^{\mu}_{cx} \cap \FFF(\theta)\neq \emptyset$, in fact \cite{bernard2017robust} explicitely constructs a minimal $\Sigma_{cx}$ element in each  Fr\'echet class $\FFF(\theta)$  whose sums are minimal in convex order in $\DDD(\mu)$, where $\mu=\sum_{1=1}^d\t_i$.

 We conclude this example specifying  the polytope $\PPP^{1.2}_{cx}$ in dimension $d=3$, that is the class of all multivariate Bernoulli distributions with sums distributed as the minimal  convex element in $\DDD(1.2)$. 
The  extremal points of $\PPP^{1.2}_{cx}$ are $n_{cx}=\binom{3}{1}\binom{3}{2}=9$ and   they are reported in Table \ref{tab:d}.

\begin{table}[h]
	\centering
		\begin{tabular}{ccc|rrrrrrrrr}
$\xx_1$ &	$\xx_2$ &	$\xx_3$ &	$\rr_1$ & 		$\rr_2$ & 		$\rr_3$ & 		$\rr_4$ & 		$\rr_5$ & 		$\rr_6$ &$\rr_7$&$\rr_8$& $\rr_9$\\
		\hline
0 &	0 &	0 &	0 &	0 &	0&	0 &	0&	0 &0&0&0\\
1 &	0 &	0 &	0.8 &	0.8 &	0.8&	0 &	0 &	0&0&0&0\\
0 &	1 &	0 &	0 &	&	0 &	0.8 &	0.8 &	0.8&0&0&0\\
1 &	1 &	0 &0.2 &	0 &	0 &	0.2&	0 &	0&0&0&0.2\\
0 &	0 &	1 &	0&	0&	0 &	0&	0 &	0&0.8 &0.8&$0.8$\\
1 &	0 &	1 &	0 &	0.2 &0 &	0&	0.2 &	0 &0.2&0&0\\
0 &	1 &	1 &	0 &	0 &0.2&	0 &	0 &	0.2&0&0.2&0\\
1 &	1 &	1 &	0 &	0 &	0 &	0 &	0 &	0&0 &0&0\\
		\end{tabular}
	\caption{Extremal pmfs  $\PPP^{1.2}_{cx}$}
	\label{tab:d}
\end{table}

\end{example}

\subsection{Entropy}
We consider the definition of the Shannon entropy for a discrete random variable $W$ which takes value in the set $\mathcal{W}$ and is distributed according to $p:\mathcal{W} \rightarrow [0,1]$ where $p_w:=P[W=w]$. The Shannon entropy is given by:
\[
H(W)=-\sum_{w \in \mathcal{W}} p_w \log p_w
\]
with the convention $0 \log(0) = 0$.

It is easy to verify that, given a random variable $X$ with distribution $p \in \mathcal{D}_d$, the extremal random variables $\RRR_{1}, \ldots  \RRR_{n_P}$ of the polytope $\PPP(p)$ have all the same entropy, which is equal to the entropy of $X$:
\[
H(\RRR_{i})=H(X), i=1,\ldots,n_P.
\] 

Moreover, we can identify the multivariate Bernoulli variables whose distributions lie within the polytope $\PPP(p)$, and which have the maximum and minimum entropy. We recall that $\design_k=\{\xx\in \design: \sum_{i=1}^dx_i=k\}$ is the subset of $\design=\{0,1\}^d$ that contains all the $\binom{d}{k}$ binary vectors with $k$ ones and $d-k$ zeros, $k=0, 1, \ldots, d$ and $\xx_k=(\underbrace{1,\ldots,1}_{k},0,\ldots, 0)\in \design_k$ is the vector with one in the first $k$ components and zero in the remaining $d-k$ components, $k=0,1,\ldots, d$.

\begin{proposition} \label{prop:entropy}
Given a pmf $p=(p_0,p_1,\ldots,p_d) \in \mathcal{D}_d$, let  $\XX_M$  be a multivariate Bernoulli random variable with pmf $\ff_M$ defined as follows:
\[
\ff_M(\xx)=\begin{cases}
\frac{p_k}{\binom{d}{k}} \,\,\, \text{if}\,\,\, \xx \in \design_k, k=0,\ldots,d \\
0 \,\,\, \text{otherwise}, 
\end{cases}
\]
Then the following hold:
\begin{enumerate}
\item $\XX_M=\argmax_{\XX \in  \PPP(p)} H(\XX)$ 
\item $\RRR_i=\argmin_{\XX \in  \PPP(p)} H(\XX), i=1,\ldots,n_p$,
\end{enumerate}
where $\RRR_i\sim\rr_i$ and $\rr_i$ are the extremal pmfs of $\PPP(p),\,\,\, i=1,\ldots n_p$. 
\end{proposition}

\begin{proof}
Both statements can be proved by noting that for any $\XX \in \FFF_d$ we can express its  entropy as:
\[
H(\XX)=-\sum_{\xx \in \design} \ff(x) \log(\ff(x))=- \sum_{k=0}^d \sum_{\xx \in \design_k} \ff(x) \log(\ff(x)).
\]
To maximize $H(\XX)$, it is sufficient to maximize each term $-\sum_{\xx \in \design_k} \ff(x) \log(\ff(x))$, which represents the entropy restricted to $\design_k$, $k=0,\ldots,d$. It is well-known that entropy is maximized by choosing a uniform distribution, which in this case corresponds to $\frac{p_k}{\binom{d}{k}}$, $k=0,\ldots, d$. This proves the first statement of the proposition.
Similarly,  to minimize $H(\XX)$, we need to minimize $-\sum_{\xx \in \design_k} \ff(x) \log(\ff(x))$,  $k=0,\ldots,d$. The entropy restricted to $\design_k$ is minimized by choosing a Dirac delta distribution centered at any point in $\design_k$, $k=0,\ldots,d$, proving the second statement of the proposition.
\end{proof}

\begin{remark}
We observe that $\ff_M\in\FFF_d$ is the exchangeable pfm correponding to $p\in \DDD$.
\end{remark}
\section{The induced measure on $\DDD$}\label{Sec: mesure}

The next Theorem \ref{main} proves our main result that Bernoulli sums induce a Dirichlet distribution on the simplex $\DDD$. We need some preliminaries.
Since  $\FFF_d\subset \RR^{2^d}$  is the standard $2^{d-1}$-simplex,  $\DDD\subset\RR^d$ is the standard $d$-simplex and $\PPP(p)\subset \RR^{2^d}$ is a $(2^d-d-1)$-convex  polytope  we consider the  Lebesgue measure $\mathcal{L}^{2^d}$  and the Hausdorff measure $\mathcal{H}^n$ for any $n\in\{0,\ldots,2^{d-1}\}$. We recall that $\mathcal{H}^0(x)=1$, for any $x\in \RR^m$, $m\in \NN$ (a standard reference for Hausdorff measures is \cite{evans2018measure}).  The Corollary  \ref{prop:Polmes} finds the  Hausdorff measure  of $\PPP(p)$  in $\RR^{2^d-d-1}$ for any   pmf $p\in\DDD$. 
 It is well known that the Hausdorff measure of the $n$-simplex with side $p$, $\Delta_{n,\sqrt{2}p}\subset \RR^{n+1}$ is 
\begin{equation}\label{VolSimplp}
 \mathcal{H}^{n}(\Delta_{n,\sqrt{2}p})=\frac{(\sqrt{2}p)^{n}\sqrt{n+1}}{n!\sqrt{2^{n}}}=\frac{p^{n}\sqrt{n+1}}{n!}.
\end{equation}
We have the following corollary of Theorem \ref{PolGen}.
\begin{corollary}\label{prop:Polmes}
For any $p=(p_0,\ldots, p_d)\in \DDD$, it holds
\begin{equation*}\label{PolMes}
\mathcal{H}^{2^d-d-1}(\PPP(p))=\prod_{k=0}^d\mathcal{H}^{n_k}(\Delta_{n_k, \sqrt{2}p_k}),
\end{equation*}
where  $n_k=\binom{d}{k}-1$.% $\mathcal{H}^{n_j}$ is the Hausdorff measure on $\RR^{n_j}$ and $\Delta_{\sqrt{2}p_j}$ is the $n_j$-simplex with side $p_j$.

\end{corollary}
%\begin{proof}
%By construction we have
%\begin{equation}
%\begin{split}
%\PPP(p)=\{f\in \FFF_d: \, f(\boldsymbol{0})=p_0, \, \sum_{\xx\in \design_{j_k}}f(\xx)=p_{j_k}, \, k=1,\ldots h\}.
%\end{split}
%\end{equation}
%
%Let $\Delta_{n_j, \sqrt{2}p_j}=\{\xx\in \RR^{n_j}: \, \sum_{k=1}^{n_j}x_k=p_j\}$, and the map
%\begin{equation}
%\begin{split}
%\mathcal{H}:&\PPP(p)\rightarrow \times_{j\in J}\Delta_{n_j, \sqrt{2}p_j}\\
%&\ff\rightarrow ((f(\xx): \, \xx\in\design_{j_1}), \ldots, (f(\xx): \, \xx\in\design_{j_{h}}))
%\end{split}
%\end{equation}
%is an isometry and the Hausdorff measure is invariant  under isometries. The thesis follows since $\mathcal{H}^{2^d-d-1}(\times_{j\in J}\Delta_{\sqrt{2}p_j})=\prod_{j=0}^d\mathcal{H}^{n_j}(\Delta_{n_j, \sqrt{2}p_j})$.
%
%\end{proof}
We can now prove  our main result.
%\begin{definition}
%Let us consider the measurable space $(\DDD, \mathcal{B}(\DDD))$, where $\mathcal{B}(\DDD)$ is the Borel $\sigma$-algebra on $\DDD$, the measure  $\mu_s$ induced from the function  $s$  in Equation  \eqref{funsum}  is defined  by:
%\begin{equation}\label{mus}
%\mu_s(A)=\mathcal{H}^{2^d-1}(s^{-1}(A))=\int_{A}\prod_{k=0}^{d} \frac{p_k^{n_k}}{n_k!}d\mathcal{H}^{d}(p),\,\,\, A\in \mathcal{B}(\DDD).
%\end{equation}
%
%\end{definition}

\begin{theorem}\label{main}

Let  $\mu_s$ the measure  on $(\DDD, \mathcal{B}(\DDD))$, where $\mathcal{B}(\DDD)$ is the Borel $\sigma$-algebra on $\DDD$, induced from the function  $s$  in Equation  \eqref{funsum}. It holds
 \begin{equation}\label{mus}
\mu_s(A)=\mathcal{H}^{2^d-1}(s^{-1}(A))=\int_{A}\prod_{k=0}^{d} \frac{p_k^{n_k}}{n_k!}d\mathcal{H}^{d}(p),\,\,\, A\in \mathcal{B}(\DDD),
\end{equation}
where $n_k=\binom{d}{k}-1$.
Then $\mu_s$
 is a positive finite measure on $\DDD$ such that $\mu_s(\DDD)=\mathcal{H}^{2^d-1}(\FFF_d)$. The measure $\mu_s$  is absolutely continuous with respect the Hausdorff measure on $\mathcal{D}_d$.
\end{theorem}
\begin{proof}

We start proving the following
\begin{equation}\label{eqproof2}
\begin{split}
\mathcal{H}^{2^d-1}(\FFF_d)
&=\int_{\DDD}\prod_{k=0}^{d}\frac{1}{\sqrt{n_k+1}}\mathcal{H}^{n_k}(\Delta_{n_k, \sqrt{2}p_k})d\mathcal{H}^{d}(p),
\end{split}
\end{equation}

%Let $\design^-=\design\setminus\{\boldsymbol{1}_{2^d}\}$ and for any $\ff\in \FFF_d$ let $\ff^-\in \Sigma_{2^d}$ be the vector $(f(\xx),\,\,\ \xx\in \design^-)$.
We explicitely  build an isometry $\alpha^{ort}$ on $\FFF_d$  by ortonormalising the following transformation $\alpha$
\begin{equation}\label{funsum}
\begin{split}
\alpha:&\FFF_d\rightarrow \FFF_d\\
   &\ff\rightarrow  (p_k^j),
\end{split}
\end{equation}
where  $k=0,\ldots, d$, $j=1,\ldots, \binom{d}{k}$ and
\begin{equation*}
p^j_k=
\begin{cases}
\sum_{\xx\in \design_k}f(\xx), \,\,\, j=1,\\
f(\xx_k^j), \,\,\, j\neq 1.
\end{cases}
\end{equation*}
Since the resulting isometry does not modify $p_k^1$,  we write $p_k=p_k^1$, $k=0,\ldots, d$.

Let $I_{2^d}=(i_{\xx, j})_{\xx\in \design, j\in \{1,\ldots 2^d\}}$  be the $2^d$- identity matrix and let $\boldsymbol{i}_{\xx}$ be the row vector $\boldsymbol{i}_{\xx}=(i_{\xx, j}, \, j=1,\ldots 2^d)$.  Let $\xx_k$ be the first element in inverse lexicographic order of $\design_k$ and
$\aa_{\xx_k}:=\frac{1}{\sqrt{n_k+1}}{\xx_k}$.

Let $A:=I_{2^d}(\boldsymbol{i}_{\xx_k}\rightarrow \aa_{\xx_k})$ be the matrix obtained from $I_{2^d}$ by replacing the row $\boldsymbol{i}_{\xx_k}$ with the row $\aa_{\xx_k}$, $k=1,\ldots, d$.  Let  $A^{ort}$ be the  matrix obtained from $A$ by the  Gram–Schmidt  ortonormalization process. It holds $\aa^{ort}_{\xx_k}=\aa_{\xx_k}$. In fact, since  $\aa_{k}=(\boldsymbol{1}_{\design_k}(\xx), \xx\in\design)$, we have that  $\aa_{k, j}=1$ implies $\aa_{h, j}=0$  and $\boldsymbol{i}_{\xx_h, j}=0$  for any $h<k$. Thus, we have $\langle \aa_k, A_h\rangle=0$, for any row $A_h$ with $h\leq k$.
Since $A^{ort}$ is an ortonormal matrix the application 
\begin{equation*}
\begin{split}
\alpha^{ort}:& \mathcal{F}_d\rightarrow \mathcal{F}_d\\
&\ff\rightarrow \ff^{ort}=A^{ort}\ff,\\
\end{split}
\end{equation*}
 is an isometry. Then it holds
\begin{equation}\label{eqproof}
\begin{split}
\mathcal{H}^{2^d-1}(\FFF_d)&=\int_{\FFF_d}d\mathcal{H}^{2^d-1}(\ff)=\int_{\Delta_{2^d, \sqrt{2}}}d\mathcal{H}^{2^d-1}(\ff)=\\
&\int_{\Delta_{d, \sqrt{2}}}\prod_{k=0}^{d}\frac{1}{\sqrt{n_k+1}}\int_{\Delta_{n_k, \sqrt{2}p_k}}d\mathcal{H}^{n_k}(\ff^{ort }_{\design_k})d\mathcal{H}^{d}(p)\\
&=\int_{\DDD}\prod_{k=0}^{d}\frac{1}{\sqrt{n_k+1}}\mathcal{H}^{n_k}(\Delta_{n_k, \sqrt{2}p_k})d\mathcal{H}^{d}(p),
\end{split}
\end{equation}
where $\ff^{ort }_{\design_k}=(A^{ort}\ff, \ff\in \design_k).$
that is \eqref{eqproof2}.
Plugging Equation \eqref{VolSimplp} in \eqref{eqproof} we have

\begin{equation*}\label{eqproof3}
\begin{split}
\mathcal{H}^{2^d-1}(\FFF_d)
&=\int_{\DDD}\prod_{k=0}^{d} \frac{p_k^{n_k}}{n_k!}d\mathcal{H}^{d}(p).
\end{split}
\end{equation*}
Thus, for any $A\in \mathcal{B}(\DDD)$,

\begin{equation*}%\label{eqproof}
\begin{split}
\mu_s(A)=\mathcal{H}^{2^d-1}(s^{-1}(A))&=\int_{s^{-1}(A)}d\mathcal{H}^{2^d-1}(\ff)=
\int_{A}\prod_{k=0}^{d} \frac{p_k^{n_k}}{n_k!}d\mathcal{H}^{d}(p),
\end{split}
\end{equation*}

Let
$
l:\DDD\rightarrow \RR^+$ defined by $l(p)=\prod_{k=0}^{d} \frac{p_k^{n_k}}{n_k!}$. The function 
is almost surely continous on $\DDD$. 
Therefore, the measure $\mu_s$ on $\mathcal{D}_d$ defined by  Equation  \eqref{funsum}   is a positive, finite and Hausdorff  absolutely continuous measure on $\DDD$. By construction $\mu_s(\DDD)=\mathcal{L}(\FFF_d)$.

\end{proof}
\begin{remark}
We observe that Theorem \ref{main} can be generalized to any surjective map $h$ as discussed in Remark \ref{RemGen}.
\end{remark}
The function $l:\DDD\rightarrow \RR^+$ defined by $$l(p)=\prod_{k=0}^{d} \frac{p_k^{n_k}}{n_k!}$$ is the density of $\mu_s$ with respect the Hausdorff measure $\mathcal{H}^d$ on $\DDD$.
The following corollary provides a useful formula for practical computations.
\begin{corollary}\label{comput}
It holds

\begin{equation}\label{measVol}
\mathcal{H}^{2^d-1}(\FFF_d)=\sqrt{2^d}\int_{\Sigma_d}\prod_{k=0}^{d-1} \frac{p_k^{n_k}}{n_k!}dp_0\dots dp_{d-1},
\end{equation}
where $ \Sigma_d=\{\xx\in \RR^d: x_j\geq 0, j=0,\ldots, d-1, \,\,\ \sum_{k=0}^{d-1}x_k\leq1\}$, and  $n_k=\binom{d}{k}-1$.
\end{corollary}
%\begin{proof}
%We have 
%
%\begin{equation}\label{eqproof}
%\begin{split}
%\mathcal{H}^{2^d-1}(\FFF_d)&=\mu_s(\DDD)=\int_{\DDD}\prod_{k=1}^{d-1} \frac{p_k^{n_k}}{n_k!}dp_0\cdots dp_{d}=\sqrt{2^d}\int_{\Sigma_d}\prod_{k=1}^{d-1} \frac{p_k^{n_k}}{n_k!}dp_0\dots dp_{d-1},
%\end{split}
%\end{equation}
%where  $ \Sigma_{2^d}=\{\xx\in \RR^{2^d}: x_j\geq 0, j=1,\ldots, d, \,\,\ \sum_{k=0}^{d-1}x_k\leq1\}$.
%\end{proof}
\begin{remark}
Theorem \ref{main} provides a link between the Hausdorff measure of the polytopes $\PPP(p)$  and the Hausdorff measure of $\FFF_d$. It holds
$\FFF_d=\cup_{p\in \mathcal{D}_d} \PPP(p)$, and
\begin{equation*}
\mathcal{H}^{2^d-1}(\FFF_d)=N\int_{\mathcal{D}_d}\mathcal{H}^{2^d-d-1}(\PPP(p))d\mathcal{H}^d(p),
\end{equation*}
where $N=\prod_{k=0}^{d}\frac{1}{\sqrt{n_k+1}}$.

\end{remark}

The following proposition states that the Hausdorff measure of $\PPP(p)$, $p\in \DDD$,  induces the Dirichlet distribution on the simplex of discrete distibutions $\DDD$  (see \cite{book2011wang} for an overview on the Dirichlet distribution). Let  $Dirichlet(\alpha_0,\ldots, \alpha_{d})$ be the Dirichlet distributions with parameters $\alpha_0,\ldots, \alpha_{d}$.
\begin{proposition}\label{dirich}
The  density $l(p)$ normalized over the simplex $\DDD$ is the Dirichlet density with parameters $\alpha_k=\binom{d}{k}$, $k=0,\ldots, d$, on the $d$-simplex $\DDD$.
\end{proposition}
\begin{proof}
It is sufficient to observe $l(p)\sim \prod_{k=0}^{d}p_k^{n_k}$ and $n_k=\binom{d}{k}-1$ with $n_0=n_d=0$.

\end{proof}
Proposition \ref{dirich} gives a geometrical interpretation for the parameters of this Dirichlet distribution as the Hausdorff dimensions of the simplexes corresponding to each $p_j$, $j=0,\ldots,d$.

\begin{remark}
Notice that $\mathcal{H}^{2^d-1}(\FFF_d)=\frac{\sqrt{2^d}}{(2^d-1)!}$ is the normalizing constant for $l(p)$ to be the Dirichlet density.
\end{remark}

The size of the class of multivariate Bernoulli distributions the sums of which have  pmf  close to a given $p\in \DDD$ depends on the behaviour of $l(p)$ in a neighborhood of $p$. 

The next Corollary \ref{mode} explicitely provides the pmf $p^M\in\mathcal{D}_d$ that maximizes the Hausdorff measure $\mathcal{H}^{2^d-d-1}(\mathcal{P}(p))$.
\begin{corollary}\label{mode}
Let $p^M=(p_0^M,\ldots p_d^M)\in \mathcal{D}$ be such that $$p_k^M=\frac{\binom{d}{k}-1}{2^d-d-1}, \,\,\, k=0,\ldots, d$$ then
 $p^M=\text{argmax}_{p\in \DDD}\mathcal{H}^{2^d-d-1}(\mathcal{P}(p)).$

\end{corollary}
\begin{proof}
The proof follows directly from Proposition \ref{dirich} observing that $p^M$ is the mode of the Dirichlet distribution.
\end{proof}
%\begin{theorem}
%Let $p^M=(p_0^M,\ldots p_d^M)\in \mathcal{D}$ be such that $$p_j^M=\frac{\binom{d}{k}-1}{2^d-d-1},$$ then
% $p^M=\text{argmax}_{p\in \DDD}\mathcal{L}(\mathcal{P}^p).$
%
% 
%
%
%
%\end{theorem}
%\begin{proof}
%The pmf $p^M\in \mathcal{D}_d$ is the solution of the following optimization problem:
%\begin{equation}\label{opt}
%\begin{split}
%&\max_{p\in \DDD}\mathcal{L}(\PPP^p)\\
%\end{split}
%\end{equation}
%Since  $\mathcal{L}(\PPP^p)=\prod_{k=1}^{d-1} \frac{(\sqrt{2}p_k)^{n_k}\sqrt{n_k+1}}{n_k!\sqrt{2^{n_k-1}}}$, the problem \eqref{opt} is equivalent to 
%\begin{equation}
%\begin{split}
%&\max_{p}\sum_{i=1}^d(\binom{d}{k}-1)\log p_k\\
%&\text{sub}\\
%&\sum_{i=1}^dp_i^M=1\\
%&p^M_i\geq 0,\,\, i=1,\ldots, d.
%\end{split}
%\end{equation}
%By using the Lagrange multiplier we  find:
%$p_j^M=\frac{\binom{d}{k}-1}{2^d-d-1}.$ 
%\end{proof}
We name $p^M$ the maximal  pmf in $\DDD$.

\subsection{Measuring a neighbourhood of $\PPP(p)$ }

This section finds the Hausdorff measure of the Bernoulli sums whose pmf is close to a given $p\in \DDD$ according to a given metrics $d$.
In practice, using the measure $\mu_s$ we   measure a neighborhood of a pmf $p\in\DDD$.
This means  finding the Hausdorff  measure in $\FFF_d$ of the set of multivariate Bernoulli distributions $\ff$ such that $p_{\ff}$ is close to $p$. 

Formally, let $d$ be a distance  on $\DDD$ and define a neighborhood of $p$ in $\DDD$ by
\begin{equation*}\label{Ip}
I_d(p, \epsilon)=\{\tilde{p}\in \mathcal{D}_d: d(\tilde{p},p)\leq\epsilon\}, \,\,\, \epsilon >0,
\end{equation*}
and a corresponding neighborhood of $\PPP(p)$ in $\FFF_d$ as its counterimage through the map $s$ is \eqref{funsum}
 \begin{equation*}\label{Ipf}
I^{\FFF}_d(p, \epsilon)=s^{-1}(I_d(p,\epsilon))=\{\tilde{\ff}\in \FFF: d(\tilde{p},p)\leq\epsilon\},\,\,\, \epsilon >0,
\end{equation*}
where $\tilde{p}=p_{\tilde{\ff}}.$
Using   Equation \ref{mus},   the Hausdorff measure in  $\FFF_d$ of  $I^{\FFF}_d(p, \epsilon)$ is the measure $\mu_s$ of $I_d(p, \epsilon)$, and this can be found by  integration of $l(p)$ over $I_d(p, \epsilon)$.  

 Following \cite{chevallier2011law} we consider two distances on $\mathcal{D}_d$.
Given two probability measures $P$ and $Q$  on a measurable space $(E, \mathcal{A})$ the total variation distance $TV$ is defined by:
\begin{equation*}
d_{TV}(P, Q)= \sup_{A\in \mathcal{A}}|P(A)-Q(A)|,
\end{equation*}
that, as proved in Proposition 2.3 in \cite{fraiman2023quantitative} for discrete distributions $\tilde{p}, p\in \DDD$ reduces to 
\begin{equation*}
d_{TV}(\tilde{p}, p)= \frac{1}{2}\sum_{k=0}^ d|\tilde{p}_k-p_k|.
\end{equation*}
and the maximum distance $d_S$ defined by
\begin{equation*}
d_S(\tilde{p},p):=\max_{0\leq k\leq d}|\tilde{p}_k-p_k|.
\end{equation*}

\begin{proposition}\label{mudelta}
If $\mu$ is a positive measure on $\FFF_d$ the following inequalities hold:
\begin{equation*}
 \mu(I^{\FFF}_{TV}(p,\epsilon))\leq \mu(I^{\FFF}_S(p,\epsilon))
\end{equation*}
\end{proposition}
\begin{proof}
We have
\begin{equation*}
d_{TV}(\tilde{p}, p)= \sup_{A\in \mathcal{P}(\design)}|P(A)-\tilde{P}(A)|=\max_{A\in \mathcal{P}(\design)}|P(A)-\tilde{P}(A)|\geq \max_{0\leq k\leq d}|\tilde{p}_k-p_k|=d_{S}(\tilde{p}, p),
\end{equation*}
thus the following inclusion holds:
\begin{equation*}
I^{\FFF}_{TV}(p,\epsilon)\subseteq I^{\FFF}_S(p,\epsilon)
\end{equation*}
and the assert follows directly.
\end{proof}
%
%
%\begin{proposition}
%It holds
%\begin{equation}
%2d_{TV}(\ff_p, \PPP^p)\geq d_M(\ff_p, \PPP^p),
%\end{equation}
%where $p$ is the distribution of the sums of $\XX\sim \ff_p$.
%\end{proposition}
%\begin{proof}
%We have
%\begin{equation}
%d_{TV}(\ff_p, \PPP^p)=\min_{\ff_B\in \PPP^p}\frac{1}{2}\sum_{\xx\in \design}|f_p(\xx)-f^p(\xx)|.
%\end{equation}
%let us consider
%\begin{equation}\label{eq:sums}
%\begin{split}
%\sum_{\xx\in \design}|f_p(\xx)-f^p(\xx)|&=\sum_{k=0}^d\sum_{\xx\in \design^k}|f_p(\xx)-f^p(\xx)|\\
%&\geq \sum_{k=0}^d|\sum_{\xx\in \design^k}(f_p(\xx)-f_b(\xx))|=\sum_{k=0}^d|p(k)-p(k)|.
%\end{split}
%\end{equation}
%Since the latter expression in \eqref{eq:sums} does not depend on the choice of $\ff^p\in \PPP^p$ we have:
%\begin{equation}
%2d_{TV}(\ff_p, \PPP^p)\geq\sum_{k=0}^d|p(k)-p(k)|\geq\max_{0\leq k\leq d}|p(k)-p(k)|,
%\end{equation}
%and the assert is proved.
%\end{proof}
Proposition \ref{mudelta} implies 
\begin{equation*}
\mathcal{H}^{2^d-1}(I^{\FFF}_{TV}(p, \epsilon))\leq\mathcal{H}^{2^d-1}(I^{\FFF}_S(p,\epsilon)).
\end{equation*}
We conclude this section by showing an example of how to estimate $\mathcal{H}^{2^d-1}(I^{\FFF}_S(p,\epsilon))$.
By definition of  $I_S(p, \epsilon)$ it holds
  $$I_S(p, \epsilon)=\{\xx\in \RR^{d+1}: \sum_{i=0}^d x_i=1,\,\,\, \max\{p_j-\epsilon, 0\}\leq x_j\leq\min\{p_j+\epsilon, 1\},\, j=0,\ldots d\},$$
therefore  $I_S(p, \epsilon)\subseteq \DDD$ is a convex polytope.

 From Corollary \ref{comput} it follows that 
\begin{equation}\label{Volepsilon}
\mathcal{H}^{2^d-1}(I^{\FFF}_S(p, \epsilon))=\mu_s(I_S(p, \epsilon))= \sqrt{2^d}\int_{\Sigma_S(p, \epsilon)}\prod_{j=0}^d  \frac{p_k^{n_k}}{n_k!}dp_0\dots dp_{d-1},
\end{equation}
where $$\Sigma_S(p, \epsilon)=\{\xx\in \RR^{d}: \sum_{i=0}^{d-1} x_i\leq1,\,\,\, \max\{p_j-\epsilon, 0\}\leq x_j\leq\min\{p_j+\epsilon, 1\},\, j=0,\ldots d-1\}.$$
We do not have a closed analytical form to compute $\mu_s(I_S(p, \epsilon))$, but using \eqref{Volepsilon}
we can find an   estimate $\hat{\mu}_s(I_S(p, \epsilon))$ of ${\mu}_s(I_S(p, \epsilon))$  by

\begin{equation*}%\label{approxmeasVol}
\hat{\mu}_s^{d}(I_S(p, \epsilon))=\hat{\m{H}}^{2^d-1}(I^{\FFF}_S(p, \epsilon)),
\end{equation*}
where $\hat{\m{H}}^{2^d-1}(I^{\FFF}_S(p, \epsilon))$ is an estimate  of ${\m{H}}^{2^d-1}(I^{\FFF}_S(p, \epsilon))$  
computed as
\begin{equation*}\label{approxmeasVol}
\hat{\m{H}}^{2^d-1}(I^{\FFF}_S(p, \epsilon))=\hat{E}_U[\prod_{j=0}^{d}\mathcal{H}^{n_j}(\Delta_{p_j})]\mathcal{H}^{d}(I_S(p, \epsilon))=\frac{\sum_{j=1}^{N}\prod_{j=0}^{d}\mathcal{H}^{n_j}(\Delta_{\hat{p}_j})}{N}\hat{\m{H}}^{d}(I_S(p, \epsilon)),
\end{equation*}
where $\hat{p}_j, \, j=1,\ldots, N$ are uniformly extracted from $I_S(p, \epsilon)$, the expectation $E_U$ is relative to a uniform distribution on the simplex,  and $\hat{\m{H}}(I_S(p, \epsilon))$ is computed using package volesti \cite{volesti} which  uses a  random-walk-based method to provide uniform samples from a given convex polytope.

%%%%%%%%%%%%%%%%%%%%%%%%%%%%%%%%%%%%%%%%%%%%%%%%%%%%%%%%%%%%%%%%%%%%%%%%%%%%%%%%%%%%%%%%%%%%%%%%%%%%%%%%%%%%%%%%%%%%%%%%%%%%%%%%%%%%%%%%%%%%%%%%%%%%%%%%%%%%%%%%%%%%%%%%%%%%%%%%%%%%%%%%%%%%%%%%%%%%

\section{The binomial distribution}\label{Sec:Binom}
%This section wonts to anlayse the class of multivariate Bernoulli distributions whose sums are close to the binomial distribution with $p=1/2$. This means to analyse the Bernoulli distributions close to the polytope $\PPP^B$.

This section focuses on the Bernoulli structure behind the discrete distribution corresponding to the most important independence model:  the binomial distribution.

Let $b(\t)\in \DDD$ be the pmf of the  binomial distribution with parameters $\t$ and $d$ ($B(\t,d)$) and let $\PPP({b(\t)})=\{\ff\in \FFF: p_ {\ff}=b({\t})\}$. From Theorem \ref{PolGen}
its  extremal points are 
\begin{eqnarray*}
f_B^{\sigma}(\xx)=\begin{cases}
%(1-\t)^d\,\,\, \text{if}\,\,\, \xx=(0,\ldots, 0) \\
\binom{d}{k}\t^k(1-\t)^{d-k} \,\,\, \text{if}\,\,\, \xx=\xx_k^{\sigma_k},\\
%q^d  \,\,\, \text{if} \,\,\,\xx=(1,\ldots, 1)\\
0 \,\,\, \text{otherwise}, 
\end{cases}
\end{eqnarray*}
where  $\sigma=(\sigma_0, \ldots, \sigma_k, \ldots, \sigma_d)$, $\sigma_k=1,\ldots, \binom{d}{k}$, $k=0,\ldots, d$. Since the binomial distributions have full support on $\{0,\ldots, d\}$ from Corollary \ref{npoints} the number of extremal points is $n_b:=n_p=\prod_{k=0}^d \binom{d}{ k}$.

The class of binomial distributions describes a parametrical curve on the simplex $\DDD$, given by $b(\t)=(b_0(\t),\ldots, b_d(\t))$, $\t\in [0,1]$.
The following Proposition \label{maxBin} proves that the density $l$ restricted to the binomial class is a concave function in the parameter space $[0,1]$ and maximal for $\t=1/2$, i.e. $l(\t)=\mathcal{H}^{2^d-d-1}(\PPP({b(\t)})$ is maximal for $\t=1/2$.
\begin{proposition}\label{maxBin}
The map 
\begin{equation}\label{lconc}
\begin{split}
l:&[0,1]\rightarrow \DDD\\
&\t\rightarrow b(\t),
\end{split}
\end{equation}
is a concave finction in $\t$ and 
$$
\text{argmax}_{\t\in [0,1]}\mathcal{H}^{2^d-d-1}(\PPP({b(\t)}))=\frac{1}{2}.
$$
\end{proposition}
\begin{proof}
We have
\begin{equation*}
\mathcal{H}^{2^d-d-1}(\PPP(b(\t)))=\prod_{k=0}^d\frac{(\binom{d}{k}\t^k(1-\t)^{d-k})^{n_k}\sqrt{n_k+1}}{n_k!},
\end{equation*}
thus
\begin{equation*}
\begin{split}
\log(\mathcal{H}^{2^d-d-1}(\PPP(b(\t))))&=\log{\prod_{k=0}^d\frac{(\binom{d}{k}\t^k(1-\t)^{d-k})^{n_k}\sqrt{n_k+1}}{n_k!}}\\
&=\log{\prod_{k=0}^d\frac{(\binom{d}{k})^{n_k}\sqrt{n_k+1}}{n_k!}}+\log{\prod_{k=0}^d (\t^k(1-\t)^{d-k})^{n_k}}.
\end{split}
\end{equation*}
It is sufficient to find the maximum of $f(\t)=\log  \prod_{k=0}^{d}(\t^k(1-\t)^{d-k})^{n_k}$. Straightforward computations lead to
\begin{equation*}
f'(\t)=\sum_{k=0}^d \frac{n_k}{\t(1-\t)}(k-d\t)=\sum_{k=0}^{d\t^-} \frac{n_k}{\t(1-\t)}(k-d\t)+\sum_{k=d\t^+}^{d} \frac{n_k}{\t(1-\t)}(k-d\t),
\end{equation*}
where $d\t^-$ is the largest integer smaller  than $d\t$ and  $d\t^+$ is the smallest  integer bigger than $d\t$. 
$f'(\t)=0$ iff $$\sum_{k=0}^{d\t^-}\frac{n_k}{\t(1-\t)}(k-d\t)=\sum_{k=d\t^+}^{d} \frac{n_k}{\t(1-\t)}(k-d\t)$$ and, since $n_k=n_{d-k}$ this is true iff $\t=1/2$. If $\t>1/2$ we have $f'(\t)<0$ and if $\t<1/2$ we have $f'(\t)>0$, and the maximum is reached on $\t=1/2$.
\end{proof}

\begin{remark}
Notice that the symmetric binomial distribution is the mean of the Dirichlet distribution $Dirichlet(n_0+1,\ldots, n_{d}+1)$  on the simplex $\DDD$.
\end{remark}

The following proposition proves that if the dimension $d$ increases, the pmf $b(1/2)$ converges to the distibution $p^M$.

\begin{proposition}\label{Binsim}
Let $b(1/2)$ be the pmf of the $B(1/2, d)$ and $p^M$ is the maximal polytope pmf in $\FFF_d$. We have 
\begin{equation*}
\lim_{d\rightarrow \infty}d_S({b(1/2)},p^M)=0,
\end{equation*}
and
\begin{equation*}
\lim_{d\rightarrow \infty}\mathcal{H}^{2^d-d-1}\PPP({b(1/2)})=l^M,
\end{equation*}
where $l^M=\mathcal{H}^{2^d-d-1}(\PPP(p^M)).$
\end{proposition}
\begin{proof}
We have
\begin{equation*}
|b(1/2)_k-p^M_k|=\frac{|2^d(\binom{d}{k}-1)-(2^d-d-1)\binom{d}{k}|}{2^d(2^d-d-1)}=\frac{|(d+1)\binom{d}{k})-2^d|}{2^d(2^d-d-1)}.
\end{equation*}
Since $\binom{d}{k}$ is maximal for $k=\frac{d-1}{2}$ and $k=\frac{d+1}{2}$ if $d$ is odd, we have
\begin{equation*}
\max_k|b(1/2)_k-p^M_k|=\frac{|(d+1)\frac{d!}{\frac{d-1}{2}!\frac{d+1}{2}!}-2^d|}{2^d(2^d-d-1)},
\end{equation*}
that converges to $0$ as $d$ goes to $\infty$. Similarly $\max_k|b(1/2)_k-p^M_k|$  converges to $0$ as $d$ goes to $\infty$  if $d$ is even, since  $\binom{d}{k}$ is maximal for $k=\frac{d}{2}$.

Since $\lim_{d\rightarrow \infty}d_S({b(1/2)},p^M)=0$ implies $\lim_{d\rightarrow \infty}|{b(1/2)}_k-p_k^M|=0$   for any $k\in \{0,\ldots, d\}$. We also have $\lim_{d\rightarrow \infty}d_E({b(1/2)}, p^M)=0$, where $E$ is the usual Euclidean norm and therefore $l(b(1/2))=\mathcal{H}^{2^d-d-1}(\mathcal{P}({b(1/2)}))$ converges  to the maximum $l^M=l(p^M)$ of the density $l(p)$. 

\end{proof}

Since the Bernoulli distribution is close to the maximal pmf $p^M$ the  Hausdorff measure  of $\PPP(b(1/2))$  is close to the maximal one  both in low and   high dimension. As a consequence we expect that for a given $\epsilon$, $ {\mathcal{H}^{2^d-1}}(I^{\FFF}_S({b(1/2)}, \epsilon))$ converges to the maximal one.  The following Theorem  proved in \cite{chevallier2011law}   provides asymptotic  lower bounds for the size in $\FFF_d$ of $I^{\FFF}_S({b(1/2)}, \epsilon)$ and shows that its normalized Hausdorff measure goes to one when $d$ increases.
\begin{theorem}{\cite{chevallier2011law}}
There exists a constant $A$ such that for all positive integers $d$ and
all positive numbers $\epsilon$,
\begin{equation*}
\mu(I^{\FFF}_S({b(1/2)}, \epsilon))\geq 1- \frac{A\sqrt{d}}{\epsilon^2 2^{d-1}}
\end{equation*}
and
%\begin{equation}\label{TV}
%\mu(\{\ff_p\in \FFF: \sup_{I\subseteq\{1,\ldots,d\}}|p(I)-p^B(I)|\leq\epsilon\})\geq 1- \frac{A{d}^{5/2}}{\epsilon^2 2^{d-1}}.
%\end{equation}
%Credo che la \ref{TV} sia equivalente a total variation distance: 
\begin{equation*}
\mu(I^{\FFF}_{TV}({b(1/2)}, \epsilon))\geq 1- \frac{A{d}^{5/2}}{\epsilon^2 2^{d-1}},
\end{equation*}
where $\mu$ is the normalized Hausdorff measure on the probability simplex $\mathcal{F}_d$.

\end{theorem}
 In \cite{fraiman2023quantitative} (Remark 2, Section 5) it is shown that even for moderate $d$, the lower bound of $\mu(I^{\FFF}_{S}({b(1/2)}, \epsilon))$ is close to one, meaning that the distributions of sums of Bernoulli random variables that are not close to the binomial $b(1/2)$ pmf are rare.
Using Equations \ref{measVol} and \ref{approxmeasVol}  we can find the Hausdorff  measure of $I^{\FFF}_{S}({b(1/2)}, \epsilon)$ even for small $d$, where the asymptotic result can not be applied. 
The binomial class of pmfs  $b(\t)$ with $\t\neq 1/2$ is not close to  $b(1/2)$ even in high dimension. Notice that from Proposition \ref{lconc} it follows that the closer  $\t$ is  to $1/2$ the higher the size of the corresponding polytope $\PPP(b(\t))$ is. We conclude with a mention of another  important discrete distribution, the Poisson-binomial distribution. It  is the law of the sum of independent and not identically distributed Bernoulli variables (see e.g.  \cite{tang2023poisson}  and \cite{boland1983reliability} for an example of its use in applications). The Poisson-binomial distribution with parameter $\t$, $PB(\t)$, with $\t=(\t_1,\ldots, \t_d)$,  is usually far from $b(1/2)$, as for example if $\sum_{i=1}^d\frac{\t_i}{d}\neq 1/2$, where $\t_i$ are the means of the independent Bernoulli variables. Even if $\sum_{i=1}^d\frac{\t_i}{d}= 1/2$, \cite{ehm1991binomial} proved that in general $b(\t)$ is not close to $b(1/2)$, see Remark 3) in \cite{fraiman2023quantitative}.

\subsection{Moments and correlation}
This section studies the range of dependence spanned by the Bernoulli vectors with  symmetric  binomial sums and show how far we can be from independence.

For any $\XX\sim \ff\in \PPP({b(1/2)})$, we can find the sharp bounds for the cross moments by applying Proposition \ref{prop:mom}. We have
\begin{equation*}
\frac{1}{2^d}\leq E[X_{j_1}\cdots X_{j_k}]\leq \sum_{n=k}^{d}\binom{d}{k}\frac{1}{2^d}.
\end{equation*}
in particular we have a necessary condition on the mean vector $\t$ for   $\PPP({b(1/2), \t})\neq \emptyset$  is
\begin{equation*}\label{diseqb}
\frac{1}{2^d}\leq \t_i\leq 1-\frac{1}{2^d},\,\,\, i=1,\ldots, d,
\end{equation*}
and the bounds for $\t_i$ are sharp.

When $\mathcal{P}(b(1/2), \t)\neq 0$
it is possible to find the analytical extremal points of the polytope at least in low dimension, as shown in the following Example \ref{EdiffMeans} where the polytope $\PPP(b(1/2), \t=(1/4, 2/4, 3/4))$ is considered. 

\begin{example}\label{EdiffMeans}

We consider an example in dimension $d=3$. 
  Table \ref{tab:exmedie} reports the extremal points of the convex polytope $\PPP(b(1/2),\t= (1/4, 2/4, 3/4))$.
\begin{table}[h]
	\centering
		\begin{tabular}{ccc|rrrrrrrrr}
$\xx_1$ &	$\xx_2$ &	$\xx_3$ &	$\rr_1$ & 		$\rr_2$ & 		$\rr_3$ \\
		\hline
0 &	0 &	0 &	1/8 &	1/8&	1/8\\
1 &	0 &	0 &	0 &	0&	1/8\\
0 &	1 &	0 &	0 &	1/8&	0 \\
1 &	1 &	0 &1/8 &	0 &	0\\
0 &	0 &	1 &	3/8&	2/8&	2/8 \\
1 &	0 &	1 &	0 &	1/8 &0\\
0 &	1 &	1 &	2/8 &	2/8&3/8\\
1 &	1 &	1 &	1/8&	1/8 &	1/8 \\
		\end{tabular}
	\caption{Extremal pmfs   $\PPP(b(1/2), \t=(1/4, 2/4, 3/4))$}
	\label{tab:exmedie}

\end{table}

Table \ref{tab:bounds} shows the maximum and minimum moments of the polytope, that are reached on the extremal points. Looking at the second order moments, one can see that there are positively correlated non-symmetric Bernoulli variables  with symmetric binomial sums.

\begin{table}[!ht]
    \centering
    \begin{tabular}{|l|l|l|l|l|l|}
    \hline
        $X_1$ &$ X_2$ & $X_3$ & Order & Min & Max \\ \hline
        0 & 0 & 0 & 0 & 1 & 1 \\ \hline
        1 & 0 & 0 & 1 & 0.25 & 0.25 \\ \hline
        0 & 1 & 0 & 1 & 0.5 & 0.5 \\ \hline
        1 & 1 & 0 & 2 & 0.125 & 0.25 \\ \hline
        0 & 0 & 1 & 1 & 0.75 & 0.75 \\ \hline
        1 & 0 & 1 & 2 & 0.125 & 0.25 \\ \hline
        0 & 1 & 1 & 2 & 0.375 & 0.5 \\ \hline
        1 & 1 & 1 & 3 & 0.125 & 0.125 \\ \hline
    \end{tabular}
\caption{Moments bounds for  $\PPP(b(1/2), \t=(1/4, 2/4, 3/4))$}
	\label{tab:bounds}
\end{table}

\end{example}

\subsection{Entropy}

The Shepp-Olkin entropy monotonicity conjecture proved in \cite{hillion2019proof} asserts that if $\XX$ has independent components $X_i$  with means $\t_i, \,\,\, i=1,\ldots, d$, the entropy $H(\t)$  of their sum $S=\sum_{j=1}^dX_j$, that is a function of the parameters $\t=(\t_1, \ldots, \t_d)$,  is non-decreasing in $\t$  if all $\t_j\leq 1/2$.  In \cite{harremoes2001binomial} the author proves that the binomial distribution $b(\frac{\mu}{d})\in\DDD$, $\mu=\sum_{i=1}^d\t_i$ is the maximal entropy distribution in the class of Poisson-binomial distributions $PB(\t)$. Our  Proposition \ref{maxBin} proves  that the case $\t_i=1/2$, i.e. the symmetric binomial distribution, corresponds to the Polytope $\PPP(b(\t))$ with maximal Hausdorff measure in the class of binomial distributions.
Here, we prove that the symmetrical binomial distribution   is the distribution of the sum $S$ of the $d$-dimensianal Bernoulli variable $\XX=(X_1,\ldots, X_d)\in\FFF_d$ with maximal entropy.
%This section proves that the symmetric  binomial distribution is the distribution of the sum of the Bernoulli vector with maximal entropy.

\begin{proposition}
The multivariate Bernoulli random variable $U = (U_1,\ldots,U_d) \sim \ff_U$, where 
\[
\ff_U(\xx) = \begin{cases}
\frac{1}{2^d} \,\,\, \text{if}\,\,\, \xx \in \design, \\
0 \,\,\, \text{otherwise}, 
\end{cases}
\]
achieves the maximum entropy within $\FFF_d$. The sum of its components follows a symmetric binomial distribution.
\end{proposition}
\begin{proof}
It is well-known that the uniform random variable over $\design$ has the highest entropy among the class $\FFF_d$.
\end{proof}
We notice that the symmetric binomial distribution does not have the maximum entropy within  $\mathcal{D}_d$.
On the other hand, $U' \sim \ff_{U'}$, the uniform random variable taking values in the set $\{0,1,\ldots,d\}$, achieves the maximum entropy within the class $\mathcal{D}_d$. We can study the entropy over the corresponding polytope $\mathcal{P}(\ff_{U'})$ using Proposition \ref{prop:entropy}. The minimum entropy is obtained for the extremal random variables $\RRR_i, i = 1, \ldots, n_p$, while the maximum entropy is attained by the random variable $X_M$ with pmf $\ff_M$, defined as follows:
\[
\ff_M(\xx) = \begin{cases}
\frac{1}{(d+1)\binom{d}{k}} \,\,\, \text{if}\,\,\, \xx \in \design_k, k = 0, \ldots, d, \\
0 \,\,\, \text{otherwise}.
\end{cases}
\]

%\section*{Acknowledgements}

\bibliographystyle{ieeetr}
\bibliography{biblio}

\begin{thebibliography}{10}

\bibitem{martz1988bayesian}
H.~Martz, R.~Wailer, and E.~Fickas, ``Bayesian reliability analysis of series
  systems of binomial subsystems and components,'' {\em Technometrics},
  vol.~30, no.~2, pp.~143--154, 1988.

\bibitem{fontana2020model}
R.~Fontana, E.~Luciano, and P.~Semeraro, ``Model risk in credit risk,'' {\em
  Mathematical Finance}, pp.~1--27, 2020.

\bibitem{vellaisamy2001nature}
P.~Vellaisamy and A.~P. Punnen, ``On the nature of the binomial distribution,''
  {\em Journal of applied probability}, vol.~38, no.~1, pp.~36--44, 2001.

\bibitem{van2005binomial}
P.~Van Der~Geest, ``The binomial distribution with dependent bernoulli
  trials,'' {\em Journal of Statistical Computation and Simulation}, vol.~75,
  no.~2, pp.~141--154, 2005.

\bibitem{chevallier2011law}
N.~Chevallier, ``Law of the sum of bernoulli random variables,'' {\em Theory of
  Probability \& Its Applications}, vol.~55, no.~1, pp.~27--41, 2011.

\bibitem{chaganty2006range}
N.~R. Chaganty and H.~Joe, ``Range of correlation matrices for dependent
  {B}ernoulli random variables,'' {\em Biometrika}, vol.~93, no.~1,
  pp.~197--206, 2006.

\bibitem{shepp1981entropy}
L.~A. Shepp and I.~Olkin, ``Entropy of the sum of independent bernoulli random
  variables and of the multinomial distribution,'' in {\em Contributions to
  probability}, pp.~201--206, Elsevier, 1981.

\bibitem{hillion2017proof}
E.~Hillion and O.~Johnson, ``A proof of the shepp--olkin entropy concavity
  conjecture,'' {\em Bernoulli}, vol.~23, no.~4B, pp.~3638--3649, 2017.

\bibitem{hillion2019proof}
E.~Hillion and O.~T. Johnson, ``A proof of the shepp-olkin entropy monotonicity
  conjecture,'' {\em Electronic Journal of Probability}, vol.~24, p.~126, 2019.

\bibitem{fontana2018representation}
R.~Fontana and P.~Semeraro, ``Representation of multivariate {B}ernoulli
  distributions with a given set of specified moments,'' {\em Journal of
  Multivariate Analysis}, vol.~168, pp.~290--303, 2018.

\bibitem{fontana2024high}
R.~Fontana and P.~Semeraro, ``High dimensional bernoulli distributions:
  Algebraic representation and applications,'' {\em Bernoulli}, vol.~30, no.~1,
  pp.~825--850, 2024.

\bibitem{bernard2017robust}
C.~Bernard, L.~R{\"u}schendorf, S.~Vanduffel, and J.~Yao, ``How robust is the
  value-at-risk of credit risk portfolios?,'' {\em The European Journal of
  Finance}, vol.~23, no.~6, pp.~507--534, 2017.

\bibitem{shaked2007stochastic}
M.~Shaked and J.~G. Shanthikumar, {\em Stochastic orders}.
\newblock Springer, 2007.

\bibitem{evans2018measure}
L.~Evans, {\em Measure theory and fine properties of functions}.
\newblock Routledge, 2018.

\bibitem{book2011wang}
M.-L.~T. Kai Wang~Ng, Guo-Liang~Tian, {\em Dirichlet and Related Distributions:
  Theory, Methods and Applications (Wiley Series in Probability and
  Statistics)}.
\newblock Wiley, 1~ed., 2011.

\bibitem{fraiman2023quantitative}
R.~Fraiman, L.~Moreno, and T.~Ransford, ``A quantitative heppes theorem and
  multivariate bernoulli distributions,'' {\em Journal of the Royal Statistical
  Society Series B: Statistical Methodology}, vol.~85, no.~2, pp.~293--314,
  2023.

\bibitem{volesti}
V.~Fisikopoulos, A.~Chalkis, and contributors in~file inst/AUTHORS, {\em
  volesti: Volume Approximation and Sampling of Convex Polytopes}, 2020.
\newblock R package version 1.1.2.

\bibitem{tang2023poisson}
W.~Tang and F.~Tang, ``The {P}oisson binomial distribution—old \& new,'' {\em
  Statistical Science}, vol.~38, no.~1, pp.~108--119, 2023.

\bibitem{boland1983reliability}
P.~J. Boland and F.~Proschan, ``The reliability of k out of n systems,'' {\em
  The Annals of Probability}, pp.~760--764, 1983.

\bibitem{ehm1991binomial}
W.~Ehm, ``Binomial approximation to the {P}oisson binomial distribution,'' {\em
  Statistics \& Probability Letters}, vol.~11, no.~1, pp.~7--16, 1991.

\bibitem{harremoes2001binomial}
P.~Harremo{\"e}s, ``Binomial and {P}oisson distributions as maximum entropy
  distributions,'' {\em IEEE Transactions on Information Theory}, vol.~47,
  no.~5, pp.~2039--2041, 2001.

\end{thebibliography}

\end{document}